\documentclass[
final
, nomarks
]{dmtcs-episciences}

\usepackage[utf8]{inputenc}
\usepackage[T1]{fontenc}
\usepackage{subfigure}

\usepackage{amsmath}
\usepackage{longtable}
\usepackage{multirow}
\usepackage{caption}
\usepackage{enumerate}
\usepackage{float}
\usepackage{amsthm}
\usepackage{hhline}

\newtheorem{thm}{Theorem}
\newtheorem{df}[thm]{Definition}

\newtheorem{lem}[thm]{Lemma}
\newtheorem{prop}[thm]{Proposition}

\newtheorem{conj}[thm]{Conjecture}

\newcommand{\myquote}[1]{\vspace{2.5ex}
	\parbox{70ex}{\em #1}\hspace*{3.75ex}($\ast$)\\[2ex]}

\newcommand{\leafedge}{leaf-edge}
\newcommand{\leafedges}{leaf-edges}

\newcommand{\dwadobryzbior}{helpful $2$-set}

\newcommand{\dobrakonfiguracja}{helpful configuration}

\newcommand{\jedencomposed}{$1$-composed}

\newcommand{\dwacomposed}{$2$-composed}
\newcommand{\dwadobrakonfiguracja}{helpful $2$-configuration}

\newcommand{\dobrysingleton}{helpful $1$-configuration}

\newcommand{\zk}{\mathbb{Z}_k}
\newcommand{\zdwa}{\mathbb{Z}_2}
\newcommand{\ztrzy}{\mathbb{Z}_3}

\newcommand{\sprig}{sprig}
\newcommand{\sprigs}{sprigs}
\newcommand{\sprigowy}{pendant}

\newcommand{\fullyincident}{fully-incident}
\newcommand{\nonincident}{non-incident}
\newcommand{\containing}{containing}

\author{Michał Tuczyński
  \and Przemysław Wenus
  \and {Krzysztof Węsek}
	}
\title{On cordial labeling of hypertrees}
\affiliation{
  Warsaw University of Technology, Poland}
\received{2017-11-21}

\revised{2019-12-20}

\accepted{2019-6-20}

\begin{document}
\publicationdetails{21}{2019}{4}{1}{4081}
\maketitle
\begin{abstract}
 	Let $f:V\rightarrow\mathbb{Z}_k$ be a vertex labeling of a hypergraph $H=(V,E)$. This labeling induces an~edge labeling of $H$ defined by $f(e)=\sum_{v\in e}f(v)$, where the sum is taken modulo $k$. We say that $f$ is $k$-cordial if for all $a, b \in \mathbb{Z}_k$ the number of vertices with label $a$ differs by at most $1$ from the number of vertices with label $b$ and the analogous condition holds also for labels of edges.
	If $H$ admits a $k$-cordial labeling then $H$ is called $k$-cordial. The existence of $k$-cordial labelings has been investigated for graphs for decades. Hovey~(1991) conjectured that every tree $T$ is $k$-cordial for every $k\ge 2$. Cichacz, Görlich and Tuza~(2013) were first to investigate the analogous problem for hypertrees, that is, connected hypergraphs without cycles. The main results of their work are that every $k$-uniform hypertree is $k$-cordial for every $k\ge 2$ and that every hypertree with $n$ or $m$ odd is $2$-cordial. Moreover, they conjectured that in fact all hypertrees are $2$-cordial. In this article, we confirm the conjecture of Cichacz et al. and make a step further by proving that for $k\in\{2,3\}$ every hypertree is $k$-cordial.
\end{abstract}

\section{Introduction}

Graph labeling problems have been intensively studied for decades since the initiatory work of Rosa~\cite{Rosa1967}. However, much less is known about labelings of hypergraphs. In this article we consider the problem of cordial labeling of hypergraphs introduced by Cichacz, Görlich and Tuza \cite{CichaczGorlichTuza2013}.

Let $f:V\rightarrow\mathbb{Z}_k$ be a vertex labeling of a hypergraph $H=(V,E)$. This labeling induces an edge labeling of $H$ defined by $f(e)=\sum_{v\in e}f(v)$, where the sum is taken modulo $k$. 
We say that $f$ is $k$-cordial if for all $a, b \in \mathbb{Z}_k$ the number of vertices with label $a$ differs by at most $1$ from the number of vertices with label $b$ and the analogous condition holds also for labels of edges.
If $H$ admits a $k$-cordial labeling then $H$ is called $k$-cordial.
$2$-cordial labelings in the case of graphs were introduced by Cahit \cite{Cahit1987} (under the name of cordial labeling) as a weakened version of well known graceful and harmonious labelings. On the other hand, harmonious labeling and elegant labeling, other known concepts, are special cases of $k$-cordial labeling. More precisely, harmonious labeling can be defined as $|E|$-cordial labeling and elegant labeling can be defined as $|V|$-cordial labeling. For more information about various graph and hypergraph labeling problems, we refer to an extensive dynamic survey of Gallian \cite{Gallian2014}.

Since the work of Cahit, $k$-cordial labelings have been studied in numerous publications (see~\cite{Gallian2014}). Lee and Liu \cite{LeeLiu1991} and Du \cite{Du1997} proved that a complete $\ell$-partite graph is $2$-cordial if and only if at most three of its partite sets have odd cardinality. Cairnie and Edwards \cite{CairnieEdwards2000} showed that in~general the~problem of deciding whether a graph is $2$-cordial is NP-complete (the authors suggest that the~problem is NP-complete even in the class of connected graphs of diameter $2$).
Cordial labeling seems to be particularly interesting for trees.
The original motivation for Cahit to investigate cordial labeling were the~Graceful Tree Conjecture of Rosa \cite{Rosa1967} and the Harmonious Tree Conjecture of Graham and Sloane \cite{GrahamSloane1980}.
Both conjectures were (and still are) far from being solved. However, Cahit was able to prove a~common relaxation: all trees are $2$-cordial.
Hovey \cite{Hovey1991} conjectured that every tree is $k$-cordial for any $k\ge 2$ and proved this to be true for $k=3,4,5$. Note that positive solution to his conjecture would imply the Harmonious Tree Conjecture.
Hegde and Murthy \cite{HedgeMurthy2014} showed that every tree is $k$-cordial for prime $k$ provided that $k$ is not smaller than the number of vertices. Recently, it~has been showed by Driscoll, Krop and Nguyen \cite{DriscollKropNguyen2017} that the Hovey's conjecture is true for $k=6$.
For other values~of~$k$, the conjecture is still open.

Generalizations of labeling problems to hypergraphs have been studied for example for magic labelings \cite{Trenkler2001Magic,Trenkler2001SuperMagic}, antimagic labelings \cite{Cichacz2016,JavaidBhatti2012,SONNTAG2002} and sum number and integral sum number \cite{SonntagTeichert2000, SonntagTeichert2001}. 
Cichacz, Görlich and Tuza \cite{CichaczGorlichTuza2013} were first to consider $k$-cordial labelings of hypergraphs.
They provided some sufficient conditions for a hypergraph to be $k$-cordial. Their main results state that every $k$-uniform hypertree is $k$-cordial for every $k\ge 2$ and that every hypertree $H$ with $|V(H)|$ or $|E(H)|$ odd is $2$-cordial (for the definition of a hypertree, see next section). The second result is a~partial answer to the following conjecture posed in their article:
\begin{conj}[\cite{CichaczGorlichTuza2013}]\label{conj:2-cordial}
	All hypertrees are $2$-cordial.
\end{conj}
Note that one cannot hope for $2$-cordiality of forests, since the forest consisting of $2$ disjoint edges of size $2$ is not $2$-cordial.
In this article we settle Conjecture~\ref{conj:2-cordial} and make a step further by proving a~stronger statement: 
\begin{thm}\label{tw:glowne}
	All hypertrees are $k$-cordial for $k\in \{2,3\}$.
\end{thm}

Theorem~\ref{tw:glowne} generalizes results of Cahit and Hovey on $k$-cordiality of tress for $k\in \{2,3\}$.
We find this theorem to be a good indication that the following generalization of the Hovey's conjecture (in~fact, already tentatively suggested in \cite{CichaczGorlichTuza2013}) can be true:
\begin{conj}
	All hypertrees are $k$-cordial for every $k\ge 2$.
\end{conj}

We prove three cases of Theorem~\ref{tw:glowne} separately. They both follow the same method, although the~case of $k=3$  is more complicated and requires some additional notions. 
The main idea is the~following. For a given hypertree, we choose a certain configuration $S$ of $k$ edges and $k$ vertices and we inductively label $H-S$. Having this partial labeling $f$, we try to extend it to entire $H$. Sometimes this is not possible, however, it appears that then it is enough to modify $f$ on small parts of $H-S$ to succeed.

\section{Preliminaries}

A \emph{hypergraph} H is a pair $H = ( V , E )$, where $V$ is the set of vertices and $E$ is a set of non-empty subsets of $V$ called edges. 
We consider finite (not necessarily uniform) hypergraphs with edges of cardinality at least $2$. Hypergraphs with no edges are called \emph{trivial}. For a hypergraph $H$, by $n(H)$ and $m(H)$ we denote the~number of vertices and edges of $H$, respectively. 
The \emph{degree} of vertex $v$, denoted by $d(v)$, is the number of edges containing $v$. An \emph{isolated} vertex is a~vertex of degree $0$.

The \emph{incidence graph} $G_H$ of a hypergraph $H$ is a bipartite graph with the~vertex set $V(H)\cup E(H)$ and the edge set $\{ve: v\in V(H), e\in E(H), v\in e\}$. By a~\emph{cycle} in a hypergraph $H$ we understand a~cycle in~$G_H$. We say that a hypergraph $H$ is \emph{connected} if $G_H$ is connected. A \emph{hypertree} is a connected hypergraph with no cycles. Equivalently, a hypergraph $H$ is a hypertree if $G_H$ is a tree.
Observe that if two edges $e$ and $e'$ have two common vertices $v$ and $v'$ then $vev'e'v$ is a cycle. Thus our definition implies that a hypertree is a linear hypergraph, that is, every two edges can have at most one common vertex. 

Let $H=(V,E)$ be a hypergraph. For a set of vertices $W\subset V$ and a set of edges $F\subset E$, we denote $H-W=(V-W,E-\{e: \exists v \in e, v\in W \})$ and $H-F=(V,E-F)$. For a vertex $v$ or an edge $e$, we simply write $H-v$ for $H-\{v\}$ and $H-e$ for $H-\{e\}$. For a set of edges $F\subset E$, denote by $H\ominus F$ the hypergraph obtained by removing all isolated vertices from the hypergraph $H-F$.

An edge in a hypergraph is a \emph{\leafedge} if it contains at most one vertex of degree greater than $1$, otherwise we call this edge \emph{internal}. A vertex of degree one contained in a \leafedge\ is called a \emph{leaf}. Note that if $T$ is a hypertree with more than one edge, then $e\in E(T)$ is a \leafedge\ if and only if $T-e$ consists of a~hypertree and some isolated vertices. Observe that if a hypertree has more than one edge, then it has at least two \leafedges. A \emph{hyperpath} is a hypertree with at most $2$ \leafedges. A~\emph{hyperstar} is a hypertree of which every edge is a~\leafedge.

We will use the following formula for the number of edges of a hypertree.
\begin{prop}\label{liczba_krawedzi}
	If $T$ is a non-trivial hypertree then $$m(T)=1+\sum\limits_{v\in V(T)}(d(v)-1).$$
\end{prop}

By $\mathbb{Z}_k$ we denote the ring of integers modulo $k$. 
When we compare the~elements of $\mathbb{Z}_k$, we use the~order $0<\ldots <k-1$. The set of $p \times q$ matrices over $\mathbb{Z}_k$ is denoted by $M_p^q(\mathbb{Z}_k)$.

For the rest of this section, we assume that $k$ is a fixed positive integer greater than $1$. Let $f:V\rightarrow\mathbb{Z}_k$ be a vertex labeling of a hypergraph $H=(V,E)$ with $n$ vertices and $m$ edges.
This vertex labeling induces an edge labeling of $H$, also denoted by $f$, defined by $f(e)=\sum_{v\in e}f(v)$ (the sum is taken modulo $k$) for $e\in E$. We allow this abuse of notation, as it is commonly used in this topic.
Denote by $n_a(f)$ and $m_a(f)$ the numbers of vertices and edges, respectively, labeled with $a$. We say that $f$ is \emph{$k$-cordial} if for all $a,b \in \mathbb{Z}_k$ we have  $|n_a(f)-n_b(f)|\le 1$ and $|m_a(f)-m_b(f)|\le 1$. If $H$ admits a~$k$-cordial labeling then $H$ is called \emph{$k$-cordial}.

Now we will give the definition of a sprig, a key notion in this paper. Sprigs will be used for the~induction step in the main cases of our proofs. That is, in order to label a hypertree $T$, we will delete a~certain sprig, label the smaller hypergraph by induction hypothesis, and then label the vertices of the~sprig in~order to obtain a cordial labeling of $T$.

\begin{df}
	Let $H=(V,E)$ be hypergraph. A sequence $S=(e_1,\ldots,e_k;v_1,\ldots,v_k)$, where every $e_i$ is an edge of $H$ and every $v_i$ is a vertex of $H$, is called a \emph{\sprig} if
	\begin{enumerate}
		\item $v_i \in e_i$ for every $i$,
		\item $v_i$ is an isolated vertex of $H-\{e_1,\ldots,e_k\}$ for every $i$.
	\end{enumerate}
\end{df}

For a hypergraph $H$ and a sprig $S=(v_1,...,v_k; e_1,...,e_k)$,
denote by $H-S$ and $H\ominus S$ the~hypergraphs $H-\{v_1,\ldots,v_k\}$ and $H\ominus \{e_1,...,e_k\}$, respectively. A \sprig\ $S$ in a hypertree $T$ is called \emph{\sprigowy} if $T-S$ has at most one non-trivial component.

Let $H=(V,E)$ be a hypergraph and let $A\subset V$. We say that a sprig $S=(e_1,\ldots,e_k;v_1,\ldots,v_k)$
is:

\begin{itemize}
	\item \emph{\containing} $A$ 
	if every vertex of $S$ belongs to $A$,
	\item \emph{\fullyincident} with $A$ 
	if no vertex of $S$ belongs to $A$ and every edge of $S$ contains a vertex from~$A$,
	\item \emph{\nonincident} with $A$ if none of the edges of $S$ contain a vertex from $A$.
\end{itemize}

Let $S=(e_1,\ldots,e_k;v_1,\ldots,v_k)$ be a sprig.
The \emph{adjacency matrix of $S$} is a matrix $M(S)=(a_{ij})_{k\times k}$ defined by 
$a_{ij}=\left\{
\begin{array}{ll}
1 & \text{if } v_j\in e_i\\
0 & \text{otherwise}\\
\end{array}
\right.
$\hspace{-1ex}.
Let $S$ be a sprig and $M\in M_k^k(\mathbb{Z}_k)$. We call $S$ an \emph{$M$-\sprig} if $M(S)=M$. 

Let $A$ be a set of pairwise non-adjacent vertices in a hypergraph $H$. We say that a cordial labeling of $H$ is \emph{strong on $A$} if every vertex of $A$ has a different label and the numbers of edges 
intersecting with $A$ labeled with $a$ are equal for all labels $a\in \zk$.

Observe that adding isolated vertices does not change the cordiality of a hypergraph. More precisely, we have the following proposition.

\begin{prop}\label{prop:seedy}
	Let $H=(V,E)$ be a hypergraph and let $H'$ be a hypergraph obtained by adding any number of isolated vertices to $H$. If $H$ is $k$-cordial, then $H'$ is $k$-cordial. Moreover if $A\subset V$ and $H$ has a $k$-cordial labeling strong on $A$, then $H'$ has a $k$-cordial labeling strong on $A$.
	In particular, every hypergraph without edges is $k$-cordial. 
\end{prop}

Let $f$ be a labeling of a hypergraph $H=(V,E)$ and let $A\subset V$ and $x\in\zk$. We say that a~labeling $g$ of $H$ is obtained from $f$ by \emph{adding $x$ on $A$} if it is defined in the following way: \linebreak
$g(v)=\left\{
\begin{array}{ll}
f(v)+x & \text{if } v\in A\\
f(v) & \text{otherwise}\\
\end{array}
\right. 
$.

\section{$2$-cordiality of hypertrees}\label{sec:k2}

The aim of this section is to prove Theorem~\ref{tw:glowne2}, which states that all hypertrees are $2$-cordial.
The~case of $m=_2 0$ is the main part of the proof. We prove this case by induction on $m$.
The induction step goes as follows:
we delete from a hypertree $T$ a certain sprig $S$, label $T-S$ by the induction hypothesis, and then label the removed vertices. In the beginning of the section, we present sprigs that will be used in this proof. Lemma~\ref{lem:k2red} shows how to obtain a cordial labeling of $T$ from a cordial labeling of~$T-S$.
In some cases, we cannot simply extend the labeling, but we will succeed if we change the~label of~a~vertex of even degree first. We want to ensure that this operation will not unbalance the~edge labels. Therefore, we prove a stronger statement: for any vertex of even degree $u$, there exists a cordial labeling of $T$ which is strong on $\{u\}$.  Lemma~\ref{lem:k2conf} assures that vertex of even degree always exists.

In this section we will use $M_i$-\sprigs\ for $i=1,2$, where:
\vspace{2ex}

$M_1=\begin{bmatrix}
1 & 0\\
0 & 1
\end{bmatrix}$,
$M_2=\begin{bmatrix}
1 & 0\\
1 & 1
\end{bmatrix}$.
\\

\noindent For an illustration, see Figure~\ref{fig:1}.


\begin{figure}[H]		
	\centering
		\subfigure[$M_1$-sprig]{\includegraphics[scale=0.2]{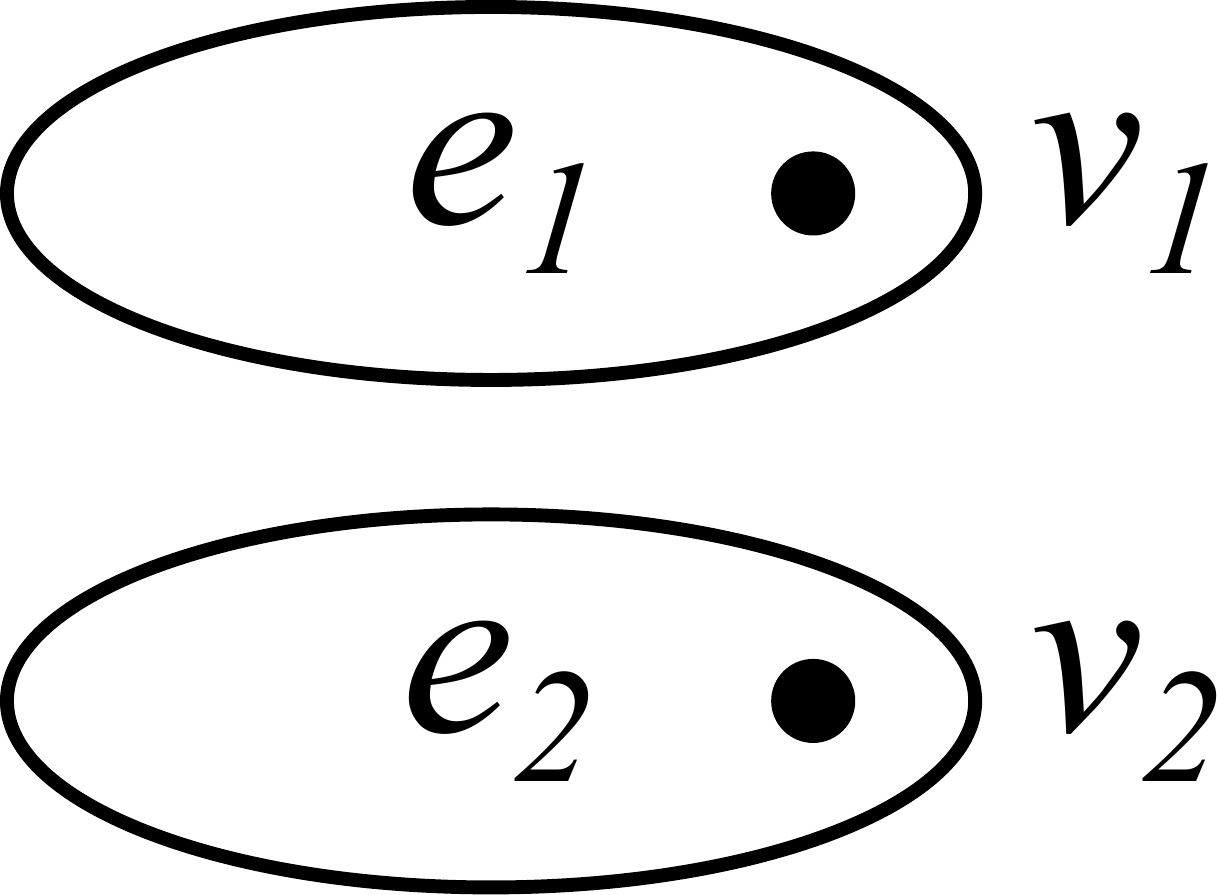}} \hfil	
		\subfigure[$M_2$-sprig]{\includegraphics[scale=0.2]{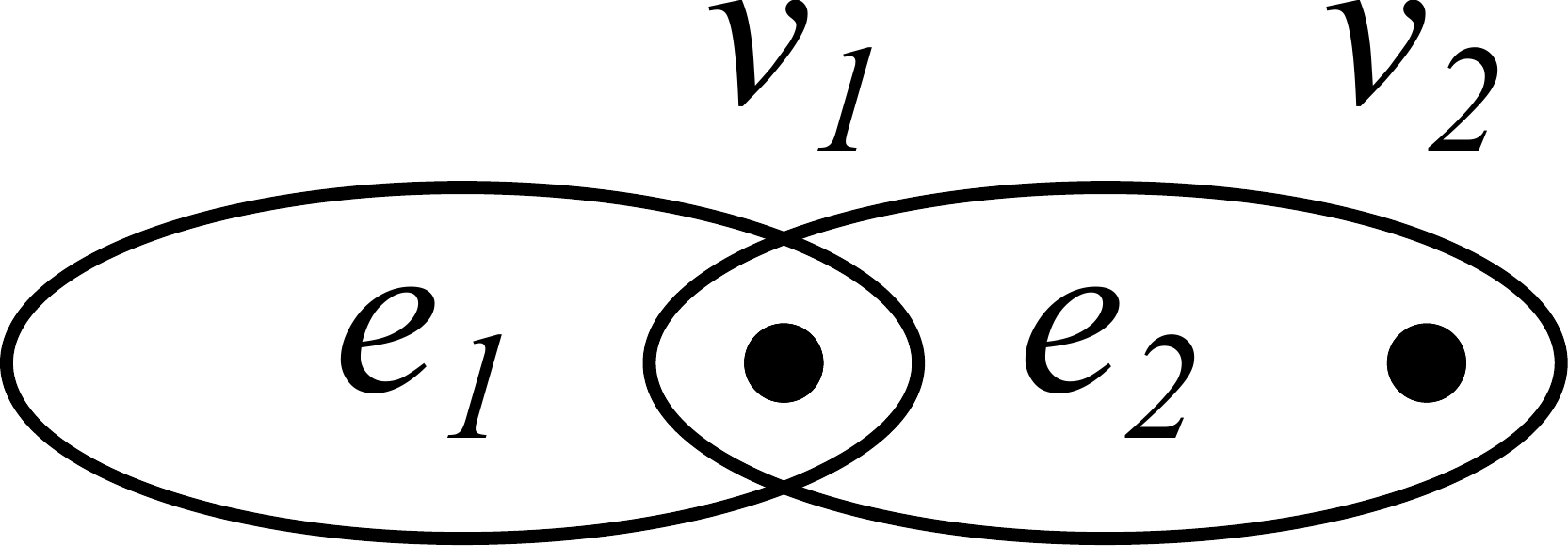}}
	\caption{Types of sprigs}\label{fig:1}		
\end{figure}

\begin{lem}\label{lem:k2red}
	
	Let $H$ be a hypergraph, $u$ a vertex of even degree in $H$ and $S$
	\begin{enumerate}\itemsep0em 
		\item an $M_1$-\sprig\ \nonincident\ or \fullyincident\ with $\{u\}$,
		\item an $M_2$-\sprig\ \nonincident\ with $\{u\}$.
	\end{enumerate}
	If $H-S$ has a $2$-cordial labeling strong on $\{u\}$, then $H$ also has $2$-cordial labeling strong on $\{u\}$.
	
\end{lem}

\begin{proof}
	Let $S=(e_1,e_2;v_1,v_2)$. Notice that $d_{H-S}(u)$ is even. Let $f$ be a $2$-cordial labeling of $H-S$ strong on $\{u\}$. For $i=1,2$ let $Y_i=e_i-\{v_1,v_2\}$ and $y_i=\sum\limits_{v\in Y_i}f(v)$. 
	\begin{description}
		\item[Case 1:] $S$ is an $M_1$-\sprig\ \nonincident\ or \fullyincident\ with $\{u\}$.
		
		If $y_1=y_2$ then we extend $f$ to a labeling of $H$ by defining $f(v_1)=0$ and $f(v_2)=1$. Then
		$f(e_1)=y_1+f(v_1)=y_1$ and $f(e_2)=y_2+f(v_2)=y_2+1$, so $f(e_1)\ne f(e_2)$ and $f$ is a $2$-cordial labeling strong on $\{u\}$.
		
		Suppose that $y_1\ne y_2$. Let $a=f(u)$. Let $f'$ be the labeling of $H-S$ obtained from $f$ by adding $1$ on $\{u\}$. Then $m_0(f')=m_0(f)$ and $m_1(f')=m_1(f)$, hence $|m_0(f')-m_1(f')|\le 1$.		
		Moreover $n_a(f')=n_a(f)-1$, $n_{a+1}(f')=n_{a+1}(f)+1$ and thus $n_{a+1}(f')-n_a(f')\in \{1,2,3\}$. We extend $f'$ to a labeling of $H$ by defining $f'(v_1)=a$ and $f'(v_2)=a$. Then 
		$$f'(e_1)=y_1+f'(v_1)=y_1+a,$$ 
		$$f'(e_2)=y_2+f'(v_2)=y_2+a$$ if $S$ is \nonincident\ with $\{u\}$ and
		$$f'(e_1)=y_1+1+f'(v_1)=y_1+1+a,$$ 
		$$f'(e_2)=y_2+1+f'(v_2)=y_2+1+a$$ if $S$ is \fullyincident\ with $\{u\}$.
		In both cases we have $f'(e_1)\ne f'(e_2)$, thus still $|m_0(f')-m_1(f')|\le 1$. Moreover now we have $n_{a+1}(f')-n_a(f')\in \{-1,0,1\}$.
		Therefore $f'$ is a $2$-cordial labeling strong on $\{u\}$.
		
		\item[Case 2:]  $S$ is an $M_2$-\sprig\ \nonincident\ with $\{u\}$.
		
		We extend $f$ to a $2$-cordial labeling of $H$ by defining $f(v_1)$ and $f(v_2)$. We have 
		$$f(e_1)=y_1+f(v_1), \qquad f(e_2)=y_2+f(v_1)+f(v_2).$$
		
		Either $(y_1,y_2)=(b,b)$ or $(y_1,y_2)=(b,b+1)$ for some $b\in\zdwa$. The values of $f(v_1)$ and $f(v_2)$ depending on $y_1$ and $y_2$ are given in Table~\ref{table_k2}.
		\begin{center}		
			\begin{longtable}{|c|c||c|c||c|c|}			
				\hline $y_1$ & $y_2$ & $f(v_1)$ & $f(v_2)$ & $f(e_1)$ & $f(e_2)$ \\  \hhline{|=|=#=|=#=|=|} b & b & 0 & 1 & b & b+1 \\ 
				\hline b & b+1 & 1 & 0 & b+1 & b\\ 
				\hline  
				\caption{}
				\label{table_k2}
			\end{longtable} 			
		\end{center}
		\vspace{-30pt}
		
		In both cases we have $f(e_1)\ne f(e_2)$, thus $f$ is a $2$-cordial labeling of $H$ strong on $\{u\}$.
		\qedhere
	\end{description}
\end{proof}

\begin{lem}\label{lem:k2conf}
	Every hypertree with an even number of edges has a vertex of even degree.
\end{lem}

\begin{proof}
	Follows from Proposition~\ref{liczba_krawedzi}.
\end{proof}

\begin{thm}\label{tw:glowne2}
	Every hypertree is $2$-cordial.
\end{thm}

\begin{proof} Let $T$ be a hypertree with $m=m(T)$. We divide the proof into two cases.
	\begin{description} 
		\item[Case 1:] $m=_20$
		
		By Lemma~\ref{lem:k2conf} every hypertree $T$ with an even number of edges has a vertex of even degree. We prove a stronger statement: If $T$ is a hypertree with an even number of edges and $u$ is a vertex of even degree in $T$, then there exists a $2$-cordial labeling of $T$ strong on $\{u\}$.

		The proof is by induction on $m$. 
		For $m=0$ the assertion obviously holds. Let $T$ be a hypertree with $m>0$ edges, $m=_20$, and let $u$ be a vertex of even degree in $T$. We will find a \sprigowy\ \sprig\ $S$, which satisfies the assumptions of Lemma~\ref{lem:k2red} and will be used in the induction step.
		
		If $m>d(u)$, then there exists a set $F$ containing two edges such that $T-F$ has at most one non-trivial component. Clearly, we can choose such two edges $e_1,e_2$ and vertices $v_i\in e_i$ for $i=1,2$ in a~way that they can be arranged into a \sprigowy\ $M_1$-\sprig\ or $M_2$-\sprig\ $S$ \nonincident\ with $\{u\}$. Notice that $d_{T\ominus S}(u)=_2 0$. Otherwise $T$ is a hyperstar with the central vertex $u$. We take as $S$ a \sprigowy\ $M_1$-\sprig\ \fullyincident\ with $\{u\}$ consisting of two edges incident with $u$ and one leaf from each of these edges. Observe that either $T\ominus S$ is the empty hypergraph or $d_{T\ominus S}(u)=_2 0$.
		
		In each case we have found a \sprigowy\ \sprig\ $S$ such that (by induction hypothesis and Proposition~\ref{prop:seedy}) $T$, $u$ and $S$ satisfy the assumptions of Lemma~\ref{lem:k2red}. Therefore, by Lemma~\ref{lem:k2red}, $T$ has a~$2$-cordial labeling strong on~$\{u\}$.

		\item[Case 2:] $m=_21$
		
		Let $e$ be a \leafedge\ in $T$. By Case~1 and  Proposition~\ref{prop:seedy}, $T-e$ has a $2$-cordial labeling $f$. Clearly, $f$ is also a $2$-cordial labeling of $T$, regardless of the induced value of $f(e)$.
		\qedhere
	\end{description}
\end{proof}

\section{3-cordiality of hypertrees}\label{sec:k3}

The case of $k=3$ needs more careful analysis than $k=2$. In this section, we extend our notation. In~the~previous section,
we used a vertex of even degree to help us to extend the labeling. 
For the~same purpose, for $k=3$ we will use two different structures. 

Let $T=(V,E)$ be a hypertree.

\begin{df}
	A set $\{u\}\subset V$ is a \emph{\dobrysingleton} if $d(u)=_3 0$.
\end{df}

\begin{df}\label{def:helpful_set}
	A set $\{u_1,u_2\}\subset V$ is a \emph{\dwadobryzbior} if $u_1$ and $u_2$ are non-adjacent, $d(u_2)=_3 2$ and $u_1$ is a leaf.
\end{df}

Notice that if $\{u_1,u_2\}$ is a \dwadobryzbior\ then $d(u_1)+d(u_2)=_3 0$. 

Sometimes we need to remove a \dwadobryzbior\ $A$ from $H$ with a \sprig\ \containing\ $A$. In order to be able to proceed by induction, we need the hypergraph obtained by removing this sprig to be a~hypertree. Before removing $A$ we have to remove some "pendant" \sprigs. Hence we introduce the~following definition.

\begin{df}\label{def:helpful_conf}
	Let $A=\{u_1,u_2\}$ be a \dwadobryzbior, where $d(u_2)=_3 2$. Denote by $P_T(A)\subset E$ the~set of edges which belong to those components of $T-u_2$ which do not contain $u_1$. We say that $A$ is a~\emph{\dwadobrakonfiguracja} if $|P_T(A)|=_3 0$.
\end{df}

Notice that if $A$ is a \dwadobrakonfiguracja\ then $(T-P_T(A))-A$ has at most one non-trivial component.

\begin{df}
	A set $A\subset V$ is called a~\emph{\dobrakonfiguracja} if it is either a~\dobrysingleton\ or a~\dwadobrakonfiguracja.
\end{df}

In order to compress the proofs in this section, we will use some 
matrix notation.
Sequences of elements from $\mathbb{Z}_3$ will be treated sometimes purely as sequences (and then denoted with round brackets) and sometimes as elements of $M_3^1(\mathbb{Z}_3)$ 
(and then denoted with square brackets). The following definitions will be used for compressing the
proofs in this section.

Let $\mathcal{P}$ and $\mathcal{D}$ be the set of all vectors from $M_3^1(\mathbb{Z}_3)$ containing exactly $3$ and $2$ distinct coordinates, respectively. 


\begin{df}
	Let $M\in M_3^3(\mathbb{Z}_3)$ and $y\in M_3^1(\mathbb{Z}_3)$. We say that $x\in \mathcal{P}$ is a~\emph{simple $M$-solution} for $y$ if $y+M x\in\mathcal{P}$.
\end{df}

\begin{df}
	Let $M\in M_3^3(\mathbb{Z}_3)$ and $y\in M_3^1(\mathbb{Z}_3)$. Consider $\mathcal{X}\subset \mathcal{D}$ and the~following conditions:
	\begin{enumerate}
		\item $|\mathcal{X}|=3$,
		\item every $x\in\mathcal{X}$ satisfies $y+M x\in\mathcal{P}$,
		\item for every $a\in\mathbb{Z}_3$ there exists $x\in\mathcal{X}$ with two coordinates equal to $a$,
		\item for every $a\in\mathbb{Z}_3$ there exists $x\in\mathcal{X}$ with no coordinate equal to $a$.
		
	\end{enumerate}
	We say that $\mathcal{X}$ is: a \emph{\jedencomposed\ $M$-solution} for $y$ if it satisfies conditions 1, 2, 3; a \emph{\dwacomposed\ $M$-solution} for $y$ if it satisfies conditions 1, 2, 4; a \emph{composed $M$-solution} for $y$ if it satisfies conditions 1, 2, 3, 4.
\end{df}

The following lemma transforms the problem of extending a partial labeling to a $3$-cordial labeling to a~problem of finding a certain $M$-solution.


\begin{lem}\label{lem:solut}
	Let $H$ be a hypergraph, $S=(e_1,e_2,e_3;v_1,v_2,v_3)$ an $M$-\sprig\ in $H$ for some $M\in M_3^3(\mathbb{Z}_3)$ and $A$ a \dobrakonfiguracja\ in $H$. Assume there exists a $3$-cordial labeling $f$ of $H-S$. For $j=1,2,3$ denote $Y_j=e_j-\{v_1,v_2, v_3\}$, $y_j=\sum\limits_{v\in Y_j}f(v)$ and $y=[y_1,y_2,y_3]^T$. Assume that one of the following conditions is satisfied:
	\begin{enumerate}
		\item $S$ is \nonincident\ or \fullyincident\ with $A$, $f$ is strong on $A$ and $x$ is a simple $M$-solution for $y$,
		\item $S$ is \containing\ $A$ and $x$ is a simple $M$-solution for $y$,
		\item $A$ is a \dobrysingleton, $S$ is \nonincident\ or \fullyincident\ with $A$, $f$ is strong on $A$ and $\mathcal{X}$ is a \jedencomposed\ $M$-solution for $y$.
		\item $A$ is a \dwadobrakonfiguracja, $S$ is \nonincident\ or \fullyincident\ with $A$, $f$ is strong on $A$ and $\mathcal{X}$ is a \dwacomposed\ $M$-solution for $y$.
	\end{enumerate}
	Then there exists a $3$-cordial labeling of $H$ strong on $A$.
\end{lem}

\begin{proof}
	First assume that $x=[x_1,x_2,x_3]^T$ is a simple $M$-solution for $y$. Let $z=[z_1,z_2,z_3]^T=y+Mx$. We extend $f$ to the labeling $f'$ of $H$ by defining $f'(v_i)=x_i$ for $i=1,2,3$. Then $f'(e_i)=z_i$ for $i=1,2,3$. Since both $x$ and $z$ are in $\mathcal{P}$, $f'$ is a $3$-cordial labeling of $H$. Clearly, if $S$ is \containing\ $A$, then $f'$ is strong on $A$. Moreover, if $f$ is strong on $A$ and $S$ is \nonincident\ or \fullyincident\ with $A$, then also $f'$ is strong on $A$. 
	
	Consider the case when $A=\{u\}$ is a \dobrysingleton\ and $\mathcal{X}$ is a \jedencomposed\ $M$-solution for $y$. Let $a=f(u)$. Choose $x=[x_1,x_2,x_3]^T\in\mathcal{X}$ with two coordinates equal to $a$ and let $b\in\mathbb{Z}_3$ be the~number missing in $x$. 
	Let $z=[z_1,z_2,z_3]^T=y+Mx$.
	Let $f'$ be the labeling obtained from~$f$ by adding $b-a$ on $A$. We extend $f'$ to the labeling of $H$ by defining $f'(v_i)=x_i$ for $i=1,2,3$. Then $f'(e_i)=z_i$ if $S$ is \nonincident\ with $A$ and  $f'(e_i)=z_i+b-a$ if $S$ is \fullyincident\ with $A$ for~$i=1,2,3$.  By the choice of $x$ we have $|n_p(f')-n_q(f')|\le1$ for all $p,q\in\mathbb{Z}_3$. Since $z\in\mathcal{P}$ then also  $|m_p(f')-m_q(f')|\le1$ for all $p,q\in\mathbb{Z}_3$. Hence $f'$ is a $3$-cordial labeling of $H$.  Moreover, since $f$ is strong on $A$ and $S$ is \nonincident\ or \fullyincident\ with $A$, we get that $f'$ is also strong on $A$.

	Now consider the case when $A=\{u_1,u_{2}\}$ is a \dwadobrakonfiguracja\ and $\mathcal{X}$ is a \dwacomposed\ \linebreak $M$-solution for $y$. 
	Let  $a$ be the element of $\mathbb{Z}_3$ different from $f(u_1)$ and $f(u_2)$. Choose $x=[x_1,x_2,x_3]^T \in \mathcal{X}$ with all coordinates distinct from $a$ and let $b\in\mathbb{Z}_3$ be the number occurring on two coordinates of~$x$.
	Moreover, let $c$ be the element of $\mathbb{Z}_3$ different from $a$ and $b$. Let $f'$ be the labeling obtained from~$f$ by adding $a-c$ on $A$. Note that $b+a-c = c$. The~rest of the proof goes as in the previous case.
\end{proof}


The plan of the rest of this section is the same as of Section~\ref{sec:k2}. The main result is Theorem~\ref{tw:glowne3}, which states that all hypertrees are $3$-cordial. Again, the most involved part is the case $m(T) =_3 0$.
The proof goes similarly to the~case of $m(T)=_2 0$ in Theorem~\ref{tw:glowne2}.
The difference is that here we use helpful configurations instead of even degree vertices.

In this section, we will use $M_i$-\sprigs\ for $i=1,\ldots,4$, where:
\vspace{2ex}

$M_1=\begin{bmatrix}
1 & 0 & 0\\
0 & 1 & 0\\
0 & 0 & 1
\end{bmatrix}$,
$M_2=\begin{bmatrix}
1 & 0 & 0\\
0 & 1 & 0\\
0 & 1 & 1
\end{bmatrix}$,
$M_3=\begin{bmatrix}
1 & 0 & 0\\
1 & 1 & 0\\
1 & 0 & 1
\end{bmatrix}$,
$M_4=\begin{bmatrix}
1 & 0 & 0\\
1 & 1 & 0\\
0 & 1 & 1
\end{bmatrix}$.

\vspace{2ex}
\noindent For an illustration, see Figure~\ref{fig:2}.

%

\begin{figure}[H]		
	\centering
	\hfil
	\subfigure[$M_1$-sprig]{\includegraphics[scale=0.17]{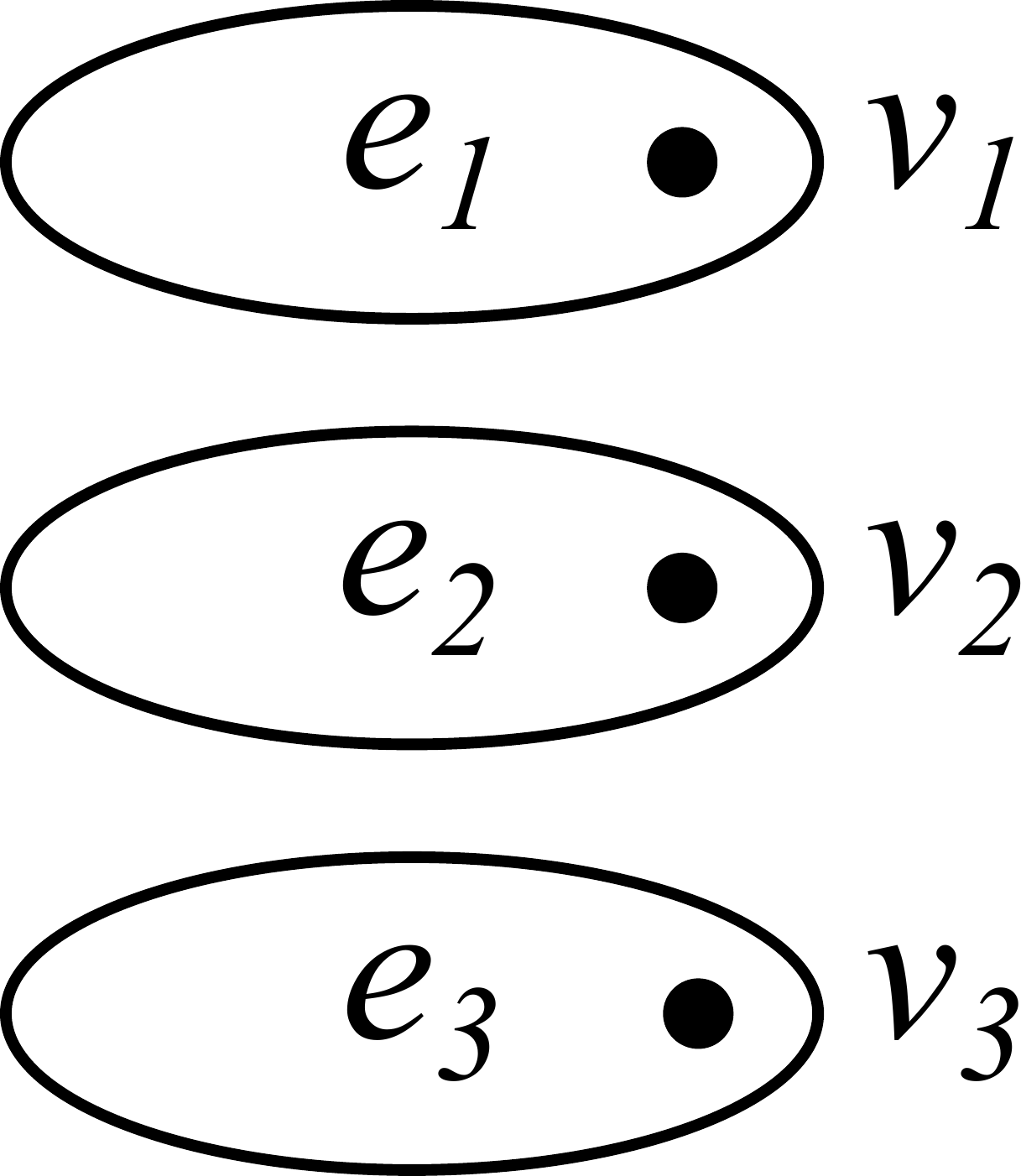}} \hfill
	\subfigure[$M_2$-sprig]{\includegraphics[scale=0.17]{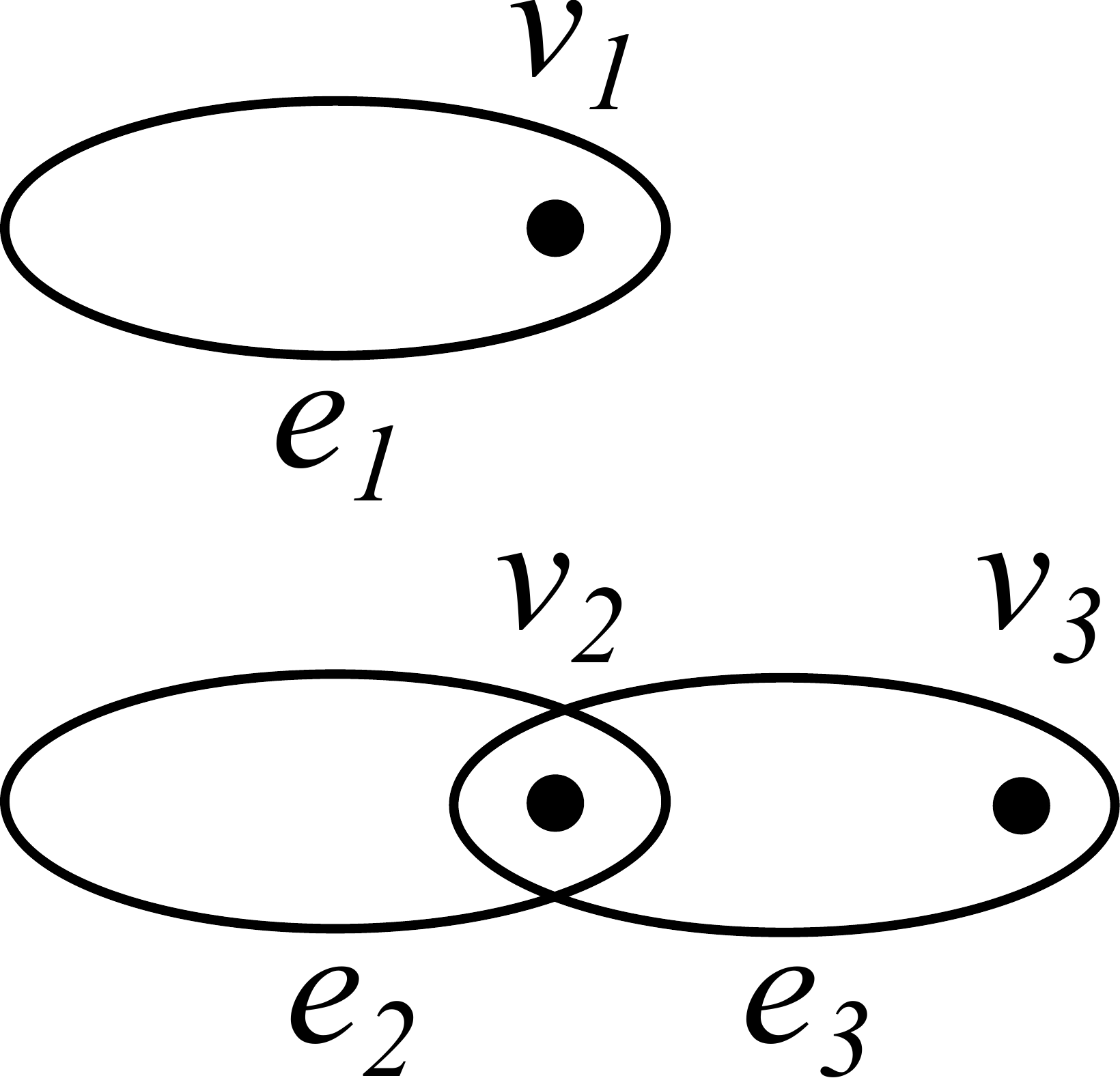}} \hfill
	\subfigure[$M_3$-sprig]{\includegraphics[scale=0.17]{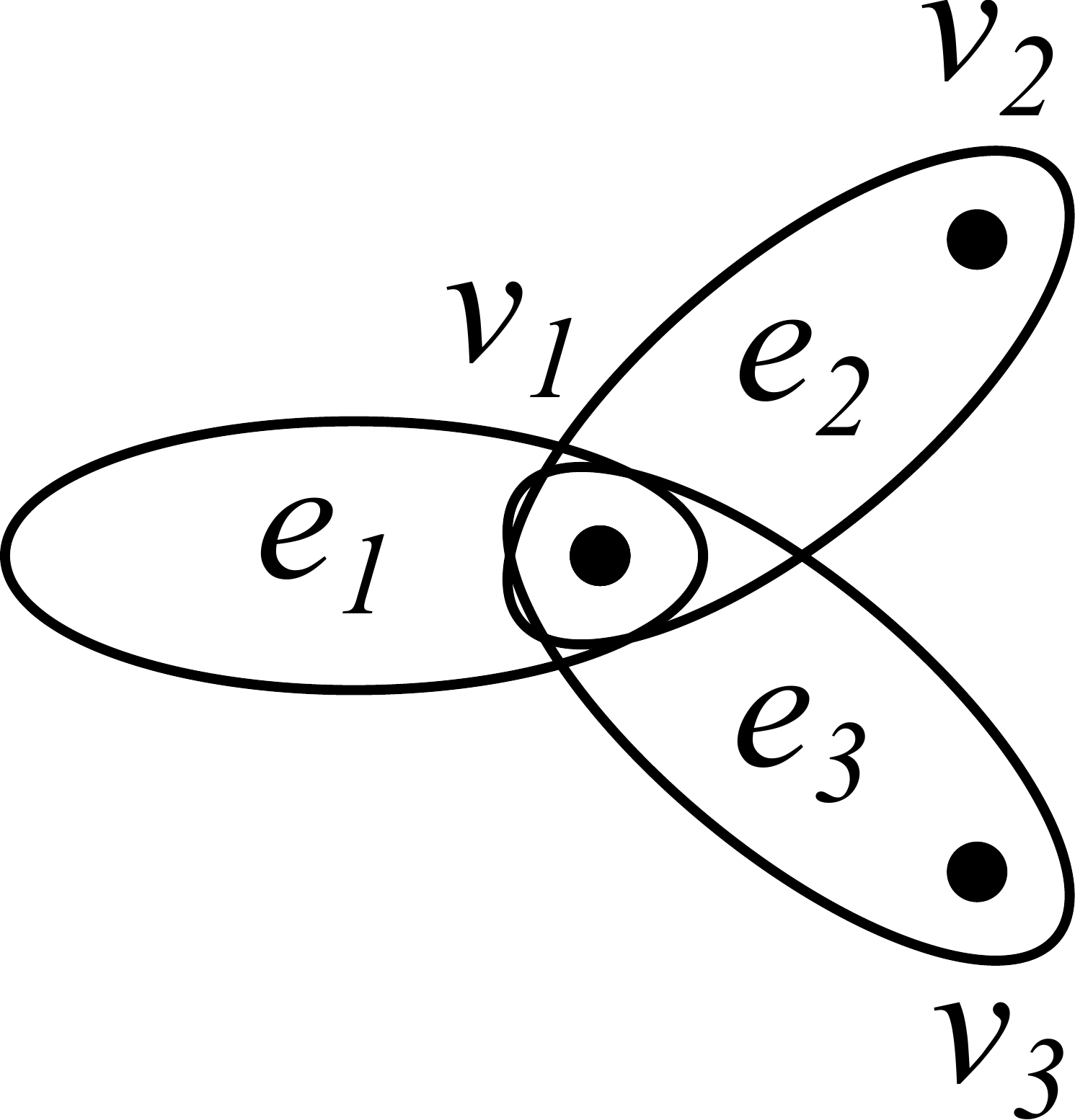}} \hfill
	\subfigure[$M_4$-sprig]{\includegraphics[scale=0.17]{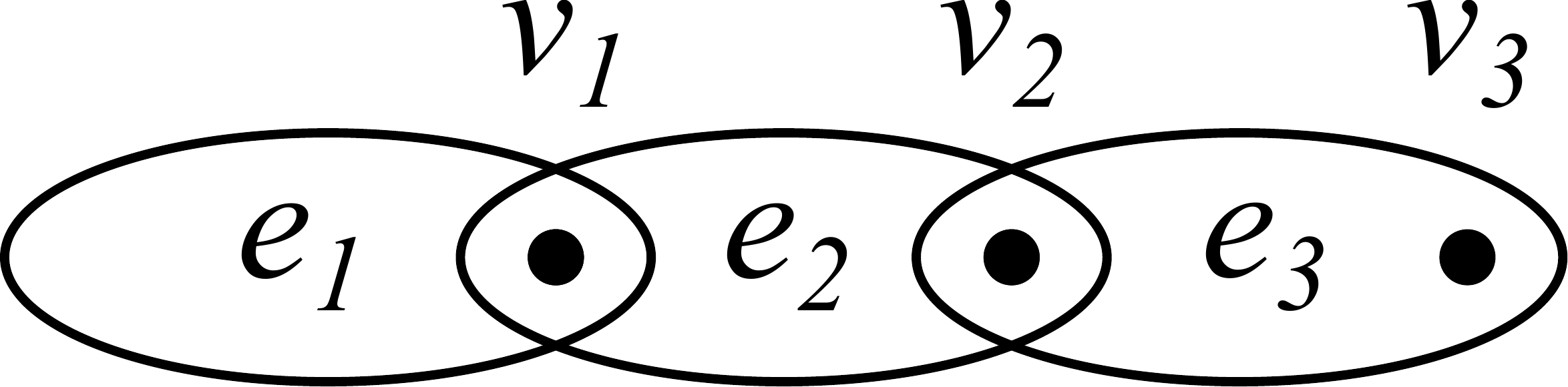}} \hfil
	\caption{Types of sprigs}\label{fig:2}		
\end{figure}

\begin{lem}\label{lem:k3red}
	Let $H$ be a hypergraph, $A$ a \dobrakonfiguracja\ in $H$ and $S$ be
	\begin{enumerate}\itemsep0em 
		\item an $M_1$-\sprig\ \nonincident\ or \fullyincident\ with $A$,
		\item an $M_2$-\sprig\ \nonincident\ with or \containing\ $A$,
		\item an $M$-\sprig\ \nonincident\ with $A$ for $M\in\{M_3,M_4\}$.
	\end{enumerate} 
	
	Assume  $H-S$ has a $3$-cordial labeling $f$. Moreover, if $S$ is \nonincident\ or \fullyincident\ with $A$, then assume $f$ is strong on $A$. 
	
	Then there exists a $3$-cordial labeling of $H$ strong on $A$.
\end{lem}

\begin{proof} 
	Let $S=(e_1,e_2,e_3;v_1,v_2,v_3)$. For $j=1,2,3$ let $Y_j=e_j-\{v_1,v_2,v_3\}$ and $y_j=\sum\limits_{v\in Y_j}f(v)$.

	By Lemma~\ref{lem:solut}, to extend the $3$-cordial labeling of $H-S$ into a $3$-cordial labeling of $H$ it suffices to find a suitable simple or composed $M_i$-solution for $y=[y_1,y_2,y_3]^T$.
	
	Let $\textbf{1}$ denote $[1,1,1]^T$. Observe that   $z\in\mathcal{P}\Leftrightarrow z+a\cdot\textbf{1}\in\mathcal{P}$ and $z\in\mathcal{D}\Leftrightarrow z+a\cdot\textbf{1}\in\mathcal{D}$ for every $a\in\mathbb{Z}_3$. Notice also that $y-b\cdot M_i\cdot\textbf{1}+M_i\cdot(x+b\cdot\textbf{1})=y+M_i\cdot x$ for every $b\in\mathbb{Z}_3$. Therefore we claim that:
	
	\myquote{
		There exists a simple (composed) $M_i$-solution for $y$ if and only if there exists a simple (composed) $M_i$-solution for $y'=y- a\cdot\textbf{1}-b\cdot M_i\cdot\textbf{1}$ for any $a,b\in\mathbb{Z}_3$.
	}

	Consider $i\in\{2,3,4\}$. Then $M_i\cdot\textbf{1}=[1,c,2]^T$ for some $c\in\{1,2\}$ depending on $i$.  By $(\ast)$ it is sufficient to find a simple or composed $M_i$-solution only for all $y'$ of the form $[0,y_2,0]$.	
	For $i=1$ we have $M_1\cdot\textbf{1}=\textbf{1}$. Hence by $(\ast)$ it is sufficient to find a simple or composed $M_1$-solution only for all $y'$ of the~form $[0,y_2,y_3]^T$. We denote $x=[x_1,x_2, x_3]^T$ and $z=y+M_i x=[z_1,z_2,z_3]^T$.
	
	\begin{description}
		\item[Case 1:] We present $M_1$-solutions for $y=[0,y_2,y_3]^T$. By symmetry we can assume that $y_2\le y_3$. Here are the corresponding formulas for $z$.	
		$$z_1=0+x_1, \qquad z_2=y_2+x_2, \qquad z_3=y_3+x_3$$
		The solutions are given in Table~\ref{table_k3_1}.		
		For each pair of values of $y_2$ and $y_3$ (which corresponds to a~single or a triple row), we give proper values of $x$. Simple solutions are presented as single rows, composed solutions are presented as triple rows. To make it easier to check, we also put the obtained values of $z$.
		\begin{center}
			\begin{longtable}{|c|c||c|c|c||c|c|c|}
				\hline $y_2$ & $y_3$ & $x_1$ & $x_2$ & $x_3$ & $z_1$ & $z_2$ & $z_3$ \\ \hhline{|=|=#=|=|=#=|=|=|}
				$0$ & $0$ & $0$ & $1$ & $2$ & $0$ & $1$ & $2$ \\
				\hline \multirow{3}{*}{0} & \multirow{3}{*}{1}
				& $0$ & $2$ & $0$ & $0$ & $2$ & $1$ \\* \cline{3-8}
				&  & $0$ & $1$ & $1$ & $0$ & $1$ & $2$ \\* \cline{3-8}
				&  & $1$ & $2$ & $2$ & $1$ & $2$ & $0$ \\* \cline{3-8}
				\hline \multirow{3}{*}{0} & \multirow{3}{*}{2}
				& $0$ & $1$ & $0$ & $0$ & $1$ & $2$ \\* \cline{3-8}
				&  & $1$ & $2$ & $1$ & $1$ & $2$ & $0$ \\* \cline{3-8}
				&  & $0$ & $2$ & $2$ & $0$ & $2$ & $1$ \\* \cline{3-8}
				\hline \multirow{3}{*}{1} & \multirow{3}{*}{1}
				& $0$ & $0$ & $1$ & $0$ & $1$ & $2$ \\* \cline{3-8}
				&  & $1$ & $1$ & $2$ & $1$ & $2$ & $0$ \\* \cline{3-8}
				&  & $2$ & $0$ & $2$ & $2$ & $1$ & $0$ \\* \cline{3-8}
				\hline  $1$ & $2$ & $0$ & $1$ & $2$ & $0$ & $2$ & $1$ \\
				\hline \multirow{3}{*}{2} & \multirow{3}{*}{2}
				& $0$ & $0$ & $2$ & $0$ & $2$ & $1$ \\* \cline{3-8}
				&  & $1$ & $0$ & $1$ & $1$ & $2$ & $0$ \\* \cline{3-8}
				&  & $2$ & $1$ & $2$ & $2$ & $0$ & $1$ \\* \cline{3-8}
				\hline
				\caption{}
				\label{table_k3_1}
			\end{longtable}
		\end{center}
		\vspace{-30pt}
		
		\item[Case 2:]
		We present $M_2$-solutions for $y=[0,y_2,0]^T$ in Table~\ref{table_k3_2}. Here are the~corresponding formulas for~$z$.	
		$$z_1=0+x_1, \qquad z_2=y_2+x_2, \qquad z_3=0+x_2+x_3$$
		\begin{center}
			\begin{longtable}{|c||c|c|c||c|c|c|}
				\hline $y_2$ & $x_1$ & $x_2$ & $x_3$ & $z_1$ & $z_2$ & $z_3$\\
				\hhline{|=#=|=|=#=|=|=|} $0$ & $1$ & $0$ & $2$ & $1$ & $0$ & $2$ \\
				\hline $1$ & $1$ & $2$ & $0$ & $1$ & $0$ & $2$ \\
				\hline $2$ & $2$ & $1$ & $0$ & $2$ & $0$ & $1$ \\
				\hline
				\caption{}
				\label{table_k3_2}
			\end{longtable}
		\end{center}
		\vspace{-30pt}
		
		\item[Case 3 ($M_3$):]
		We present $M_3$-solutions for $y=[0,y_2,0]^T$ in Table~\ref{table_k3_3}. Here are the corresponding formulas for $z$.
		$$z_1=0+x_1, \qquad z_2=y_2+x_1+x_2, \qquad z_3=0+x_1+x_3$$
		
		\begin{center}
			\begin{longtable}{|c||c|c|c||c|c|c|}
				\hline $y_2$ & $x_1$ & $x_2$ & $x_3$ & $z_1$ & $z_2$ & $z_3$\\
				\hhline{|=#=|=|=#=|=|=|} $0$ & $0$ & $1$ & $2$ & $0$ & $1$ & $2$ \\
				\hline $1$ & $1$ & $0$ & $2$ & $1$ & $2$ & $0$ \\
				\hline $2$ & $2$ & $0$ & $1$ & $2$ & $1$ & $0$ \\
				\hline
				\caption{}
				\label{table_k3_3}
			\end{longtable}
		\end{center}
		\vspace{-30pt}

		\item[Case 3 ($M_4$):]
		We present $M_4$-solutions for $y=[0,y_2,0]^T$ in Table~\ref{table_k3_4}. Here are the corresponding formulas for $z$.	
		$$z_1=0+x_1, \qquad z_2=y_2+x_1+x_2, \qquad z_3=0+x_2+x_3$$
		
		\begin{center}
			\begin{longtable}{|c||c|c|c||c|c|c|}
				\hline $y_2$ & $x_1$ & $x_2$ & $x_3$ & $z_1$ & $z_2$ & $z_3$\\
				\hhline{|=#=|=|=#=|=|=|} $0$ & $1$ & $2$ & $0$ & $1$ & $0$ & $2$ \\
				\hline $1$ & $2$ & $0$ & $1$ & $2$ & $0$ & $1$ \\
				\hline $2$ & $1$ & $0$ & $2$ & $1$ & $0$ & $2$ \\
				\hline
				\caption{}
				\label{table_k3_4}
			\end{longtable}
		\end{center}
		\vspace{-33pt}	
	\end{description}
\end{proof}

\begin{lem}\label{lem:k3conf}
	Every hypertree $T$ with $m(T)=_3 0$ has a \dobrakonfiguracja.
\end{lem}

\begin{proof}
	
	Denote $m=m(T)$. The proof is by induction on $m$. For $m=0$ the~assertion obviously holds. 
	Let $T$ be a hypertree with $m>0$ edges, $m=_30$, and suppose that there is no vertex of degree divisible by $3$ in~$T$.
	
	By Proposition~\ref{liczba_krawedzi}, there exists a vertex $v$ with $d(v)=_3 2$, denote $d=d(v)$. Let $\{e_1,\ldots,e_d\}$ be the set of edges incident with $v$. For $i=1,\ldots,d$ denote by $T_i$ the hypertree induced by $e_i$ and the components of $T-v$ intersecting with $e_i$ in $T$. Let $m_i=m(T_i)$ for $i=1,\ldots,d$.
	We have $m_1+\ldots+m_d=_3 0$. Since $d=_32$, not all $m_i=_3 1$.  We consider two cases.
	
	\begin{description}
		\item [Case 1:] $m_i=_3 0$ for some $i\in\{1,\ldots,d\}$
		
		Let $T'=T_i$. As $m_i<m$, by the induction hypothesis there exists a \dobrakonfiguracja\ $A'$ in $T'$. There are no vertices of degree divisible by $3$ in $T'$, hence $|A'|=2$. Let $u\in A'$ be the vertex which is not a leaf. If $A'$ contains a leaf from $e_i$, say $w$, then let $A=A'-\{w\}\cup\{x\}$, where $x$ is any leaf of $T$ not contained in $T_i$. Otherwise let $A=A'$. Then $A$ is a \dwadobryzbior\ in $T$. If $e_i\in P_{T'}(A')$ or if $u\in e_i$, then $P_T(A)=P_{T'}(A')\cup \big(E(T)-E(T_i)\big)$ and $|P_T(A)|=|P_{T'}(A')|+m-m_i=_30$. Otherwise $P_T(A)=P_{T'}(A')$. Therefore $A$ is a \dobrakonfiguracja\ in $T$. 
		
		\item[Case 2:] $m_i=_3 2$ for some $i\in\{1,\ldots,d\}$
		
		Let $S_1,\ldots,S_q$ be the components of $T_i\ominus \{e_i\}$. Denote $p_j=m(S_j)$ for $j=1,\ldots,q$. Note that $p_1+\ldots+p_q= m_i-1=_3 1$. Hence not all $p_j=_3 0$.
		
		Consider the case when for some $j$ we have $p_j=_3 1$. $S_j$ contains a leaf of $T$, say $x$. The set $A=\{v,x\}$ is a~\dwadobryzbior\ in $T$. Moreover, $|P_T(A)|=m - d - p_j =_3 0$. Therefore $A$ is a~\dobrakonfiguracja\ in $T$.
		
		Otherwise for some $j$, it holds that $p_j=_3 2$.
		Consider the hypertree $T'$ induced by $S_j$ and $e_i$. Note that $m > m(T') =_3 0$. By induction hypothesis there exists a \dobrakonfiguracja\ $A'$ in $T'$. There are no vertices of degree divisible by $3$ in $T'$, hence $|A'|=2$. If $A'$ contains a leaf from $e_i$, say $w$, then let $A=A'-\{w\}\cup\{x\}$, where $x$ is any leaf of $T$ not contained in $T_i$. Otherwise let $A=A'$. Then $A$ is a \dwadobryzbior\ in $T$. If $e_i\in P_{T'}(A')$ or if $u\in e_i$, then $P_T(A)=P_{T'}(A')\cup \big(E(T)-E(T_i)\big)\cup \big(E(T_i)-E(T')\big)$ and $|P_T(A)|=|P_{T'}(A')|+m-m_i+m_i-p_j-1=_30$. Otherwise $P_T(A)=P_{T'}(A')$. Therefore $A$ is a \dobrakonfiguracja\ in $T$.
	\end{description}
	
\end{proof}

\begin{thm}\label{tw:glowne3}
	Every hypertree is $3$-cordial.
\end{thm}

\begin{proof} Denote $m=m(T)$. We divide the proof into three cases.
	\begin{description}
		\item[Case 1:] $m=_{3} 0$
		
		By Lemma~\ref{lem:k3conf} every hypertree $T$ with $m=_{3} 0$ has a \dobrakonfiguracja. We prove a stronger statement: if $T$ is a hypertree with with $m=_{3} 0$ and $A$ is a \dobrakonfiguracja\ in $T$, then there exists a $3$-cordial labeling of $T$ strong on $A$.

		The proof is by induction on $m$. 
		For $m=0$ the assertion obviously holds. Let $T$ be a hypertree with $m>0$ edges, $m=_3 0$, and let $A$ be a \dobrakonfiguracja\ in $T$. 
		We will find a \sprigowy\ \sprig\ $S$, which satisfies the assumptions of Lemma~\ref{lem:k3red} and will be used in the induction step.

		First, we consider the case when $A$ is a \dobrysingleton. Let $A=\{u\}$. 
		If $m>d(u)$, then there exist a set $F$ containing three edges non-incident with $u$ such that $T\ominus F$ has at most one non-trivial component. Hence we can choose such three edges $e_1,e_2,e_3$ and vertices $v_i\in e_i$ for $i=1,2,3$ such that they can be arranged into a \sprigowy\ $M$-\sprig\ $S=(e_1,e_2,e_3;v_1,v_2,v_3)$ \nonincident\ with $A$, where $M\in\{M_1,\ldots,M_4\}$. Notice that $A$ is a \dobrakonfiguracja\ in $T\ominus S$.
		Otherwise, $m=d(u)$ and $T$ is a hyperstar with the central vertex $u$. We take as $S$ a \sprigowy\ $M_1$-\sprig\ \fullyincident\ with $A$ consisting of three edges incident with $u$ and one leaf from each of these edges. Observe that either $T\ominus S$ is the empty hypergraph or $A$ is a \dobrakonfiguracja\ in~$T\ominus S$. 
		
		Now we consider the case when $A$ is a  \dwadobrakonfiguracja. Let $A=\{u,v\}$ and $d(u)=_32$.
		If $|P_T(A)|>0$, then there exists a set $F$ containing three edges non-incident with any vertex of $A$ such that $T-F$ has at most one non-trivial component, thus there exists a \sprigowy\ $M$-\sprig\ $S$ \nonincident\ with $A$, where $M\in\{M_1,\ldots,M_4\}$. If $|P_T(A)|=0$ and $d(u)>2$ then we take as $S$ a \sprigowy\ $M_1$-\sprig\ $S$ \fullyincident\ with $A$ (consisting of three \leafedges\ incident with $u$ and suitably selected vertices). Observe that in both situations $A$ is a \dobrakonfiguracja\ in $T\ominus S$. If $|P_T(A)|=0$ and $d(u)=2$ then we take as $S$ a \sprigowy\ $M_2$-\sprig\ \containing\ $A$ (consisting of the~two edges incident with $u$, the edge incident with $v$ and suitably selected vertices).
		
		In each case we have found a \sprigowy\ \sprig\ $S$ such that (by induction hypothesis and Proposition~\ref{prop:seedy}) $T$, $A$ and $S$ satisfy the assumptions of Lemma~\ref{lem:k3red}. Therefore, by Lemma~\ref{lem:k3red}, $T$ has a~$3$-cordial labeling strong on~$A$.
		
		\item[Case 2:] $m=_{3} 1$
		
		Let $e$ be a \leafedge\ in $T$. By Case~1 and Proposition~\ref{prop:seedy}, $T-e$ has a $3$-cordial labeling $f$. Clearly, $f$ is also a $3$-cordial labeling of $T$, regardless of the induced value of $f(e)$.
		
		\item[Case 3:] $m=_{3} 2$
		
		Let $e_1$, $e_2$ be two \leafedges\ in $T$, and let $v_i$ be a leaf from $e_i$ for $i=1,2$. Let $Y_i=e_i-\{v_i\}$ for $i=1,2$. Let $T'=T-\{v_1,v_2\}$. By Case~1 and Proposition~\ref{prop:seedy}, $T'$ has a $3$-cordial labeling $f$. \linebreak For $i=1,2$ let $y_i=\sum\limits_{u\in Y_i}f(u)$. We have $n_a(f)\le n_{a+1}(f)\le n_{a+2}(f)$ for some $a\in\ztrzy$. We extend $f$ to a labeling of $T$ by defining $f(v_1)$ and $f(v_2)$. The values of $f(v_1)$ and $f(v_2)$ depending on $y_1$ and $y_2$ are given in Table~\ref{table_k3_theorem}.
		\begin{center}		
			\begin{longtable}{|c|c||c|c||c|c|}
				\hline $y_1$ & $y_2$ & $f(v_1)$ & $f(v_2)$ & $f(e_1)$ & $f(e_2)$ \\  \hhline{|=|=#=|=#=|=|} b & b & a & a+1 & a+b & a+b+1 \\ 
				\hline b & b+1 & a & a+1 & a+b & a+b+1\\ 
				\hline b & b+2 & a+1 & a & a+b+1 & a+b+2\\ 
				\hline  
				\caption{}
				\label{table_k3_theorem}
			\end{longtable} 			
		\end{center}
		\vspace{-40pt}
		\qedhere
	\end{description}
\end{proof}

\section{Conclusions}

We believe that our method can work to prove $k$-cordiality of hypertrees for larger values of $k$. However, the complication of the arguments is growing and some structures may need a special treatment. Therefore, one probably need some new ideas to work with larger values of $k$.

A hypergraph $H$ is called \emph{$d$-degenerate} if every subhypergraph $H'$ of $H$ (meaning $V(H')\subseteq V(H)$ and $E(H')\subseteq E(H)$) has a vertex of degree at most $d$.
In case of graphs, $1$-degeneracy coincides with being a forest. However, in general hypergraphs, the class of $1$-degenerated hypergraphs is much wider than the~class of hyperforests (where a hyperforest is understood as a disjoint union of hypertrees). In~this paper, we proved that all hypertrees are $2$-cordial. 
It seems to be a natural next step to determine if the~following conjecture is true:

\begin{conj}
	All $1$-degenerated connected hypergraphs are $2$-cordial.
\end{conj}

\bibliographystyle{alpha}
\bibliography{mybibfile_arxiv}
\label{sec:biblio}

\end{document}